\documentclass[10pt]{amsart}
\usepackage{amssymb,MnSymbol}
\usepackage{amsthm,amsmath}

\title[]{Subsignatures of systems}

\author{Jean-Luc Marichal}
\address{Mathematics Research Unit, FSTC, University of Luxembourg, 6, rue Coudenhove-Kalergi, L-1359 Luxembourg, Luxembourg}
\email{jean-luc.marichal[at]uni.lu}

\date{October 3, 2013}

\begin{document}

\theoremstyle{plain}
\newtheorem{theorem}{Theorem}
\newtheorem{lemma}[theorem]{Lemma}
\newtheorem{proposition}[theorem]{Proposition}
\newtheorem{corollary}[theorem]{Corollary}
\newtheorem{fact}[theorem]{Fact}
\newtheorem*{main}{Main Theorem}

\theoremstyle{definition}
\newtheorem{definition}[theorem]{Definition}
\newtheorem{example}[theorem]{Example}

\theoremstyle{remark}
\newtheorem*{conjecture}{onjecture}
\newtheorem{remark}{Remark}
\newtheorem{claim}{Claim}

\newcommand{\N}{\mathbb{N}}
\newcommand{\R}{\mathbb{R}}
\newcommand{\Q}{\mathbb{Q}}
\newcommand{\Vspace}{\vspace{2ex}}
\newcommand{\bfx}{\mathbf{x}}
\newcommand{\bfy}{\mathbf{y}}
\newcommand{\bfz}{\mathbf{z}}
\newcommand{\bfh}{\mathbf{h}}
\newcommand{\bfe}{\mathbf{e}}
\newcommand{\os}{\mathrm{os}}
\newcommand{\dd}{\,\mathrm{d}}

\begin{abstract}
We introduce the concept of subsignature for semicoherent systems as a class of indexes that range from the system signature to the
Barlow-Proschan importance index. Specifically, given a nonempty subset $M$ of the set of components of a system, we define the $M$-signature of the system
as the $|M|$-tuple whose $k$th coordinate is the probability that the $k$th failure among the components in $M$ causes the system to fail. We
give various explicit linear expressions for this probability in terms of the structure function and the distribution of the component
lifetimes. We also examine the case of exchangeable lifetimes and the special case when $M$ is a modular set.
\end{abstract}

\keywords{Semicoherent system, dependent lifetimes, Barlow-Proschan index, system signature, system subsignature.}

\subjclass[2010]{62N05, 90B25, 94C10}

\maketitle

%---------------------------------------------------------------------------------------------- Section 1
\section{Introduction}

Let $(C,\phi,F)$ be an $n$-component system (also denoted $(C,\phi)$ if no confusion arises), where $C=\{1,\ldots,n\}$ denotes the set of components, $\phi$ denotes the associated
structure function $\phi\colon\{0,1\}^n\to\{0,1\}$ (which expresses the state of the system in terms of the states of its components), and $F$
denotes the joint cumulative distribution function (c.d.f.) of the component lifetimes $T_1,\ldots,T_n$, that is,
$$
F(t_1,\ldots,t_n) ~=~ \Pr(T_1\leqslant t_1,\ldots,T_n\leqslant t_n),\qquad t_1,\ldots,t_n\geqslant 0.
$$

We assume that the system is \emph{semicoherent}, i.e., the structure
function $\phi$ is nondecreasing in each variable and satisfies the conditions $\phi(0,\ldots,0)=0$ and $\phi(1,\ldots,1)=1$. We also assume
that the c.d.f.\ $F$ has no ties, that is, $\Pr(T_i=T_j)=0$ for all distinct $i,j\in C$.

By identifying elements $(x_1,\ldots,x_n)$ of $\{0,1\}^n$ with subsets $A$ of $C$ in the usual way (i.e., setting $x_i=1$ if and only if $i\in A$), we may also regard the structure function as a set function $\phi\colon 2^C\to\{0,1\}$. For instance we can write $\phi(0,\ldots,0)=\phi(\varnothing)$ and $\phi(1,\ldots,1)=\phi(C)$.

The \emph{signature} of the system, a concept introduced first in 1985 by Samaniego \cite{Sam85} for systems whose components have continuous
and i.i.d.\ lifetimes and then recently extended to non-i.i.d.\ lifetimes (see \cite{MarMat11} and the references therein), is defined as the
$n$-tuple $\mathbf{p}=(p_1,\ldots,p_n)$, where $p_k$ is the probability that the $k$th component failure causes the system to fail. That is,
$$
p_k ~=~ \Pr(T_C=T_{k:n}),\qquad k\in\{1,\ldots,n\}\, ,
$$
where $T_C$ denotes the system lifetime and $T_{k:n}$ denotes the $k$th smallest lifetime, i.e., the $k$th order statistic obtained by rearranging
the variables $T_1,\ldots,T_n$ in ascending order of magnitude.

Interestingly, when the component lifetimes are i.i.d.\ (or even exchangeable) one can easily show that the signature $\mathbf{p}$ is independent of the distribution
function $F$. % (thus it depends only on the structure function $\phi$).
In this case the signature is often denoted by $\mathbf{s}=(s_1,\ldots,s_n)$, where $s_k=\Pr(T_C=T_{k:n})$. In
fact, Boland \cite{Bol01} showed that $s_k$ can be written explicitly in the form
\begin{equation}\label{eq:asad678}
s_k ~=~ \sum_{\textstyle{A\subseteq C\atop |A|=n-k+1}}\frac{1}{{n\choose |A|}}\,\phi(A)-\sum_{\textstyle{A\subseteq C\atop
|A|=n-k}}\frac{1}{{n\choose |A|}}\,\phi(A)\, .
\end{equation}
Being independent of $F$, the $n$-tuple $\mathbf{s}$ is a purely combinatorial object associated with the structure function $\phi$. Due to this feature, $\mathbf{s}$ is often referred to as the \emph{structural signature} of the system.

In the general nonexchangeable case the signature $\mathbf{p}$ may of course depend on $F$. In this case, it is then often referred to as the \emph{probability signature} of the system. Marichal and Mathonet \cite{MarMat11} recently showed that $p_k$ can be written explicitly in the form
\begin{equation}\label{eq:asghjad678}
p_k ~=~ \sum_{\textstyle{A\subseteq C\atop |A|=n-k+1}}q(A)\,\phi(A)-\sum_{\textstyle{A\subseteq C\atop |A|=n-k}}q(A)\,\phi(A)\, ,
\end{equation}
where the function $q\colon 2^C\to [0,1]$, called the \emph{relative quality function} associated with $F$, is defined by
\begin{equation}\label{eq:we56wz7}
q(A) ~=~ \Pr\Big(\max_{i\in C\setminus A}T_i<\min_{i\in A}T_i\Big)\, .
\end{equation}
That is, for every subset $A$ of $C$, the number $q(A)$ is the probability that the best $|A|$ components are precisely those in $A$.

Thus the general formula (\ref{eq:asghjad678}) reduces to Boland's formula (\ref{eq:asad678}) whenever $q(A)$ reduces to $1/{n\choose |A|}$ for every $A\subseteq C$, for instance when the component lifetimes are exchangeable. Formulas (\ref{eq:asghjad678}) and (\ref{eq:we56wz7}) also show how the distribution function $F$ is encoded in the probability signature $\mathbf{p}$ through the relative quality function $q$.

Since its introduction the concept of signature proved to be a very useful tool in the analysis of semicoherent systems, especially for the
comparison of different system designs and the computation of the system reliability (see, e.g., \cite{Sam07} for the i.i.d.\ case and
\cite{MarMat11,MarMatWal} for the general dependent case).

The \emph{Barlow-Proschan importance index} of the system, another useful concept introduced first in 1975 by Barlow and
Proschan~\cite{BarPro75} for systems whose components have continuous and independent lifetimes and then extended to the general dependent case
in \cite{Iye92,MarMat}, is defined as the $n$-tuple $\mathbf{I}_{\mathrm{BP}}$ whose $j$th coordinate is the probability that the failure of
component $j$ causes the system to fail, that is,
$$
I_{\mathrm{BP}}^{(j)} ~=~ \Pr(T_C=T_j)\, .
$$

Just as for the signature, when the component lifetimes are i.i.d.\ (or even exchangeable) the index $\mathbf{I}_{\mathrm{BP}}$ is also independent of the function $F$. It is then called the \emph{structural importance index} and denoted $\mathbf{b}=(b_1,\ldots,b_n)$, where
$b_j=\Pr(T_C=T_j)$. An explicit expression for $b_j$ in terms of the structure function values is given by
\begin{equation}\label{eq:6s7df5}
b_j ~=~ \sum_{A\subseteq C\setminus\{j\}} \frac{1}{n\,{n-1\choose |A|}}\,\Delta_j\phi(A)\, ,
\end{equation}
where $\Delta_j\phi(A)=\phi(A\cup\{j\})-\phi(A)$ for every $A\subseteq C\setminus\{j\}$. Marichal and Mathonet \cite{MarMat} extended this formula to the general nonexchangeable case into
\begin{equation}\label{eq:a76d5s}
I_{\mathrm{BP}}^{(j)} ~=~ \sum_{A\subseteq C\setminus\{j\}} q_j(A)\,\Delta_j\phi(A)\, ,
\end{equation}
where, for every component $j\in C$, the function $q_j\colon 2^{C\setminus\{j\}}\to [0,1]$, that we shall call the \emph{relative quality function of component $j$}, is defined by
$$%\begin{equation}\label{eq:a46asad}
q_j(A) ~=~ \Pr\Big(\max_{i\in C\setminus A}T_i=T_j<\min_{i\in A}T_i\Big)\, .
$$%\end{equation}
That is, for every component $j\in C$ and every subset $A$ of $C\setminus\{j\}$, the number $q_j(A)$ is the probability that the components that are better than component $j$ are precisely those in $A$. For instance, when $n=4$ we have
\begin{eqnarray*}
q_2(\{1,3\}) &=& \Pr(T_4<T_2<\min\{T_1,T_3\})\\
&=& \Pr(T_4<T_2<T_1<T_3)+\Pr(T_4<T_2<T_3<T_1)\, .
\end{eqnarray*}
%More generally, (\ref{eq:a46asad}) can be rewritten as
%\begin{equation}\label{eq:ds653sdfsd}
%q_j(A) ~=~ \sum_{\textstyle{\sigma\in\mathfrak{S}_n{\,}:{\,}\{\sigma(n-|A|+1),\ldots,\sigma(n)\}=A\atop\sigma(n-|A|)=j}}\Pr(T_{\sigma(1)}<\cdots <T_{\sigma(n)})\, ,
%\end{equation}
%where $\mathfrak{S}_n$ denotes the set of permutations on $C$.

By definition we have $\Delta_j\phi(A)\in\{0,1\}$ for every $j\in C$ and every $A\subseteq C\setminus\{j\}$. Moreover, we have $\Delta_j\phi(A)=1$ if and only if $\phi(A)=0$ and $\phi(A\cup\{j\})=1$, which means that component $j$ is critical with respect to subset $A$. Formula (\ref{eq:a76d5s}) then shows that $I_{\mathrm{BP}}^{(j)}$ is the sum of function $q_j$ over all subsets $A\subseteq C\setminus\{j\}$ for which $j$ is critical.

The important concepts of signature and Barlow-Proschan index motivate the introduction of the following more general concept. Let $M$ be a
nonempty subset of the set $C$ of components and let $m=|M|$. We define the \emph{$M$-signature} of the system as the $m$-tuple
$\mathbf{p}_M=(p_M^{(1)},\ldots,p_M^{(m)})$, where $p_M^{(k)}$ is the probability that the $k$th failure among the components in $M$ causes the system to
fail. That is,
$$
p_M^{(k)} ~=~ \Pr(T_C=T_{k:M}),\qquad k\in\{1,\ldots,m\}\, ,
$$
where $T_{k:M}$ denotes the $k$th smallest lifetime of the components in $M$, i.e., the $k$th order statistic obtained by rearranging the
variables $T_i$ ($i\in M$) in ascending order of magnitude. A \emph{subsignature} of the system is an $M$-signature for some $M\subseteq C$.

Clearly, when $M=C$ the $M$-signature reduces to the standard signature $\mathbf{p}$, which shows that the signature is a particular subsignature. At the
opposite, when $M$ is a singleton $\{j\}$ the $M$-signature reduces to the $1$-tuple $\mathbf{p}_{\{j\}}=(p_{\{j\}}^{(1)})$, where $p_{\{j\}}^{(1)} =
\Pr(T_C=T_j)$ is the $j$th coordinate of the Barlow-Proschan index $\mathbf{I}_{\mathrm{BP}}$. Thus, the subsignatures define a class of $2^n-1$ indexes that range from the
standard signature (when $M=C$) to the Barlow-Proschan index (when $M$ consists of a single component).

\begin{remark}\label{rem:1}
The concept of $M$-signature is particularly relevant when $M$ is a subset of potentially unreliable components. Consider for instance a large system whose components are rather reliable except two of them, $i,j\in C$, which are vulnerable. Then it may be informative to compute the probability $p_{\{i,j\}}^{(1)}$ (resp.\ $p_{\{i,j\}}^{(2)}$) that the first (resp.\ the second) failure among these two components causes the system to fail.
\end{remark}

In this paper we provide various explicit linear expressions for subsignatures. More precisely, considering the concept of subsignature as a simultaneous generalization of the concepts of signature and Barlow-Proschan index, we provide linear expressions for subsignatures which are simultaneous generalizations of formulas (\ref{eq:asghjad678}) and (\ref{eq:a76d5s}). We also provide linear expressions for subsignatures in terms of the signed domination function of the system (recall that the signed domination function defines the coefficients of the multilinear expression of the structure function). This is done in Section 2. In Section 3 we investigate the special case when the component lifetimes are exchangeable. Just as formulas (\ref{eq:asghjad678}) and (\ref{eq:a76d5s}) then reduce to formulas (\ref{eq:asad678}) and (\ref{eq:6s7df5}), respectively, we show how the general formulas obtained in Section 2 can also be particularized to this special case. These particularized formulas then show that, under the assumption of exchangeable lifetimes, the subsignatures do not depend on the distribution function $F$. For every nonempty subset $M$ of $C$, we then denote the $M$-signature $\mathbf{p}_M$ by $\mathbf{s}_M$ and naturally call it \emph{structural $M$-signature}. Finally, in Section 4 we examine the case when $M$ is a modular set and show how the $M$-signature is then related to the signature of the corresponding module.

\begin{remark}
In this paper we focus on the concept of subsignature as a mathematical generalization of the concepts of signature and Barlow-Proschan index and stress mainly on the theoretical and logical construction of the linear expressions that we provide for the subsignatures. Applications of the concept of subsignatures will be presented in another paper.
\end{remark}

%---------------------------------------------------------------------------------------------- Section 2
\section{Main results}

In this section we provide and discuss various explicit linear expressions for the probability $p_M^{(k)}=\Pr(T_C=T_{k:M})$. We start with expressions in terms
of the functions $q_j$ and $\Delta_j\phi$, thus generalizing formula (\ref{eq:a76d5s}).

\begin{theorem}\label{thm:8sf6}
For every nonempty set $M\subseteq C$ and every $k\in\{1,\ldots,m\}$, we have
\begin{equation}\label{eq:main}
p_M^{(k)} ~=~ \sum_{\textstyle{A\subseteq C\atop |M\setminus A|=k}}\,\sum_{j\in M\setminus A}\, q_j(A)\,\Delta_j\,\phi(A) ~=~ \sum_{j\in
M}\,\sum_{\textstyle{A\subseteq C\setminus\{j\}\atop |M\setminus A|=k}}\, q_j(A)\,\Delta_j\,\phi(A)\, .
\end{equation}
\end{theorem}

\begin{proof}
Let $S_n$ be the set of permutations on $C$ and, for every $\sigma\in S_n$, let $\omega_{\sigma}$ be the event $(T_{\sigma(1)}<\cdots < T_{\sigma(n)})$. Since the c.d.f.\ $F$ has no ties, the events $\omega_{\sigma}$ ($\sigma\in S_n$) form a partition almost everywhere (a.e.)\ of the sample space $\Omega=\left[0,+\infty\right[^n$.

For every $A\subseteq C$ such that $|M\setminus A|=k$ and every $j\in M\setminus A$, define the event
$$
E_A^j ~=~ \Big(\max_{i\in C\setminus A}T_i=T_j<\min_{i\in A}T_i\Big).
$$
These events form a partition a.e.\ of $\Omega$. Indeed, for every $\sigma\in S_n$, there exists a unique $i\in C$ such that $\sigma(i)\in M$ and $|\{\sigma(1),\ldots,\sigma(i)\}\cap M|=k$. We then have $\omega_{\sigma}\subset E_A^j$ if and only if $j=\sigma(i)$ and $A = \{\sigma(i+1),\ldots,\sigma(n)\}$.

Moreover, for every $A\subseteq C$ such that $|M\setminus A|=k$ and every $j\in M\setminus A$, we have $E_A^j \subset (T_C=T_{k:M})$ if and only if $\phi(A)=0$ and $\phi(A\cup\{j\})=1$, that is, $\Delta_j\phi(A)=1$. Otherwise, if $\Delta_j\phi(A)=0$, then $E_A^j\cap (T_C=T_{k:M})=\varnothing$.

We then have
$$
(T_C=T_{k:M}) ~\stackrel{\text{a.e.}}{=} \bigcup_{\textstyle{A\subseteq C {\,}:{\,} |M\setminus A|=k\atop j\in M\setminus A {\,}:{\,} \Delta_j\phi(A)=1}}E_A^j
$$
and hence
$$
p_M^{(k)} ~=~ \sum_{\textstyle{A\subseteq C\atop |M\setminus A|=k}}\,\sum_{j\in M\setminus A}\, \Pr(E_A^j)\,\Delta_j\,\phi(A),
$$
which proves the first expression in (\ref{eq:main}). The second one can be obtained by permuting the sums in the first expression.
\end{proof}

For instance in the special case when $M=\{i,j\}$, formulas (\ref{eq:main}) reduce to
\begin{eqnarray*}
p_{\{i,j\}}^{(1)} &=& \sum_{A\subseteq C\setminus\{i,j\}}\Big(q_i(A\cup\{j\})\,\Delta_i\,\phi(A\cup\{j\})+q_j(A\cup\{i\})\,\Delta_j\,\phi(A\cup\{i\})\Big){\,},\\
p_{\{i,j\}}^{(2)} &=& \sum_{A\subseteq C\setminus\{i,j\}}\Big(q_i(A)\,\Delta_i\,\phi(A)+q_j(A)\,\Delta_j\,\phi(A)\Big){\,}.
\end{eqnarray*}

\begin{example}\label{ex:7sdfds}
Consider a $3$-component system whose structure function is given by
$$
\phi(x_1,x_2,x_3) ~=~ (x_1\amalg x_2)\, x_3 ~=~ x_1x_3+x_2x_3-x_1x_2x_3{\,},
$$
where $\amalg$ is the coproduct operation defined by $x\amalg y=1-(1-x)(1-y)$. For such a system, we have for instance
\begin{eqnarray*}
p_{\{1,3\}}^{(1)} &=&
%\sum_{\textstyle{A\subseteq\{1,2,3\}\atop |\{1,3\}\setminus A|=1}}\,\sum_{j\in \{1,3\}\setminus A}\, q_j(A)\,\Delta_j\,\phi(A) ~=~
q_1(\{3\})+q_3(\{1\})+ q_3(\{1,2\})\\
&=& \Pr(T_2<T_1<T_3)+\Pr(T_2<T_3<T_1)+\Pr(T_3<T_1<T_2)+\Pr(T_3<T_2<T_1).
\end{eqnarray*}
\end{example}

Formula (\ref{eq:main}) shows that $p_M^{(k)}$ is a sum of $q_j(A)$ over certain subsets $A$ and the components $j$ in $M$ that are critical with respect to these subsets. In particular $p_M^{(k)}$ is a partial sum of terms of the form $\Pr(T_{\sigma(1)}<\cdots <T_{\sigma(n)})$, where $\sigma$ is a permutation on $C$.

When $M$ is a singleton $\{j\}$ we see immediately that (\ref{eq:main}) reduces to (\ref{eq:a76d5s}). When $M=C$, formula (\ref{eq:main}) provides the following new explicit expressions for the $k$th coordinate $p_k$ of the probability signature:
\begin{equation}\label{eq:w675re}
p_k ~=~ \sum_{\textstyle{A\subseteq C\atop |A|=n-k}}\,\sum_{j\in C\setminus A}\, q_j(A)\,\Delta_j\,\phi(A) ~=~ \sum_{j\in
C}\,\sum_{\textstyle{A\subseteq C\setminus\{j\}\atop |A|=n-k}}\, q_j(A)\,\Delta_j\,\phi(A)\, .
\end{equation}
Contrary to formula (\ref{eq:asghjad678}), these formulas give an expression for $p_k$ as a partial sum of terms of the form $\Pr(T_{\sigma(1)}<\cdots <T_{\sigma(n)})$.

\begin{example}
Consider the structure defined in Example~\ref{ex:7sdfds} and let us compute $p_1$. On the one hand, Eq.~(\ref{eq:asghjad678}) provides the expression
$$
p_1 ~=~ 1- \Pr(T_2<T_1<T_3)- \Pr(T_2<T_3<T_1)- \Pr(T_1<T_2<T_3)- \Pr(T_1<T_3<T_2).
$$
On the other hand, Eq.~(\ref{eq:w675re}) provides the partial sum $p_1=\Pr(T_3<T_1<T_2)+\Pr(T_3<T_2<T_1)$.
\end{example}

Interestingly, we have the following link between the subsignatures and the Barlow-Proschan index. For every nonempty subset $M\subseteq C$, we have
\begin{equation}\label{eq:probIpk}
\sum_{k=1}^mp_M^{(k)} ~=~ \Pr(T_C=T_j~\mbox{for some}~j\in M) ~=~ \sum_{j\in M}I_{\mathrm{BP}}^{(j)}{\,}.
\end{equation}
Using either (\ref{eq:a76d5s}) or (\ref{eq:main}), we obtain immediately the following expression for probability (\ref{eq:probIpk}).

\begin{corollary}
For every nonempty set $M\subseteq C$, we have
$$
\Pr(T_C=T_j~\mbox{for some}~j\in M) ~=~ \sum_{A\subseteq C}\,\sum_{j\in M\setminus A}\, q_j(A)\,\Delta_j\,\phi(A).
$$
\end{corollary}

If probability (\ref{eq:probIpk}) is strictly positive, then we can express the \emph{normalized $M$-signature} $p_M^{(k)}/\sum_{\ell =1}^mp_M^{(\ell)}$ as the conditional
probability
\begin{equation}\label{eq:probnorm}
\frac{p_M^{(k)}}{\sum_{\ell =1}^mp_M^{(\ell)}} ~=~ \Pr(T_C=T_{k:M}\mid T_C=T_j~\mbox{for some}~j\in M)\, .
\end{equation}

Formula~(\ref{eq:main}) expresses $p_M^{(k)}$ as a weighted sum of functions $\Delta_j\phi$ ($j\in M$). The following result yields an alternative expression for the probability $p_M^{(k)}$ as a weighted sum of function $\phi$.

\begin{corollary}\label{cor:8sdf6}
For every nonempty set $M\subseteq C$ and every $k\in\{1,\ldots,m\}$, we have
\begin{equation}\label{eq:sfd5}
p_M^{(k)} ~=~ \sum_{j\in M}\,\sum_{\textstyle{A\subseteq C\atop |(M\setminus A)\cup\{j\}|=k}}\, (-1)^{|\{j\}\setminus A|}\,
q_j(A\setminus\{j\})\,\phi(A)\, .
\end{equation}
\end{corollary}

\begin{proof}
The right-hand side of (\ref{eq:sfd5}) can be written as
\begin{eqnarray*}
\lefteqn{\sum_{j\in M}\,\sum_{\textstyle{A\subseteq C,~A\ni j\atop |M\setminus A|=k-1}}\, q_j(A\setminus\{j\})\,\phi(A)-\sum_{j\in M}\,\sum_{\textstyle{A\subseteq C,~A\not\ni j\atop |M\setminus A|=k}}\, q_j(A)\,\phi(A)}\\
&=& \sum_{j\in M}\,\sum_{\textstyle{A\subseteq C\setminus\{j\}\atop |M\setminus A|=k}}\, q_j(A)\,\phi(A\cup\{j\})-\sum_{j\in
M}\,\sum_{\textstyle{A\subseteq C\setminus\{j\}\atop |M\setminus A|=k}}\, q_j(A)\,\phi(A)\, ,
\end{eqnarray*}
which is precisely the right-hand side of (\ref{eq:main}).
\end{proof}

We now provide an alternative linear expression for the probability $p_M^{(k)}$ which generalizes formula (\ref{eq:asghjad678}). This expression is a difference of two partial sums of function $\phi$, weighted by probabilities.

For every nonempty
set $M\subseteq C$, define the set functions $q_M^{\downarrow}\colon 2^C\setminus\{\varnothing\}\to\R$ and $q_M^{\uparrow}\colon
2^C\setminus\{C\}\to\R$ by
$$
q_M^{\downarrow}(A) ~=~ \sum_{j\in M\cap A}q_j(A\setminus\{j\}) ~=~ \Pr\Big(\exists\, j\in M:\max_{i\in C\setminus A}T_i<T_j=\min_{i\in
A}T_i\Big)
$$
and
$$
q_M^{\uparrow}(A) ~=~ \sum_{j\in M\setminus A}q_j(A) ~=~ \Pr\Big(\exists\, j\in M:\max_{i\in C\setminus A}T_i=T_j<\min_{i\in A}T_i\Big)\, ,
$$
respectively.

\begin{corollary}\label{cor:8sdf67}
For every nonempty set $M\subseteq C$ and every $k\in\{1,\ldots,m\}$, we have
\begin{equation}\label{eq:main2}
p_M^{(k)} ~=~ \sum_{\textstyle{A\subseteq C\atop |M\cap A|=m-k+1}}\, q_M^{\downarrow}(A)\,\phi(A)-\sum_{\textstyle{A\subseteq C\atop |M\cap
A|=m-k}}\, q_M^{\uparrow}(A)\,\phi(A)\, .
\end{equation}
\end{corollary}

\begin{proof}
By (\ref{eq:main}) we have
\begin{eqnarray*}
p_M^{(k)} &=& \sum_{\textstyle{A\subseteq C\atop |M\setminus A|=k}}\,\sum_{j\in M\setminus A}\, q_j(A)\,\Delta_j\,\phi(A)\\
&=& \sum_{\textstyle{A\subseteq C\atop |M\setminus A|=k}}\,\sum_{j\in M\setminus A}\, q_j(A)\,\phi(A\cup\{j\})-\sum_{\textstyle{A\subseteq
C\atop |M\setminus A|=k}}\,\sum_{j\in M\setminus A}\, q_j(A)\,\phi(A)\\
&=& \sum_{\textstyle{A\subseteq C\atop |M\setminus A|=k-1}}\,\sum_{j\in M\cap A}q_j(A\setminus\{j\})\,\phi(A)-\sum_{\textstyle{A\subseteq C\atop
|M\setminus A|=k}}\,q_M^{\uparrow}(A)\,\phi(A)\, ,
\end{eqnarray*}
which completes the proof.
\end{proof}

We end this section by providing an explicit linear expression for the probability $p_M^{(k)}$ in terms of the \emph{signed domination function} of the system \cite{BarIye88} (or equivalently, the M\"obius transform of the structure
function \cite[Sect.~1.5]{Ram90}). Recall that the signed domination function of the system is the set function $d\colon 2^C\to\R$ which gives the coefficients of the unique
multilinear expression of the structure function, that is,
$$
\phi(\bfx) ~=~ \sum_{A\subseteq C}d(A)\,\prod_{i\in A}x_i\, .
$$
The conversion formulas between $d$ and $\phi$ are given by
\begin{equation}\label{eq:76sdds}
d(A) ~=~ \sum_{B\subseteq A}(-1)^{|A|-|B|}\,\phi(B)\quad\mbox{and}\quad\phi(A) ~=~ \sum_{B\subseteq A}d(B)\, .
\end{equation}
A very simple linear expression for $p_M^{(k)}$ in terms of the signed domination function is presented in the following theorem.

\begin{theorem}\label{thm:sa8d7}
For every nonempty set $M\subseteq C$ and every $k\in\{1,\ldots,m\}$, we have
$$
p_M^{(k)} ~=~ \sum_{\textstyle{A\subseteq C\atop |M\cap A|\leqslant m-k+1}}d(A)\,\Pr\Big(T_{k:M}=\min_{i\in A}T_i\Big)\, ,
$$
or equivalently,
$$
p_M^{(k)} ~=~ \sum_{A\subseteq C}d(A)\,\Pr\Big(T_{k:M}=\min_{i\in A}T_i\Big)\, .
$$
\end{theorem}

\begin{proof}
By substituting the second formula of (\ref{eq:76sdds}) in (\ref{eq:main2}) and then permuting the resulting sums, we obtain
\begin{eqnarray*}
p_M^{(k)} &=& \sum_{\textstyle{A\subseteq C\atop |M\cap A|=m-k+1}}\, q_M^{\downarrow}(A)\,\sum_{B\subseteq A}d(B)-\sum_{\textstyle{A\subseteq C\atop |M\cap A|=m-k}}\, q_M^{\uparrow}(A)\,\sum_{B\subseteq A}d(B)\\
&=& \sum_{\textstyle{B\subseteq C\atop |M\cap B|\leqslant m-k+1}}d(B)\,\sum_{\textstyle{A\supseteq B\atop |M\cap A|= m-k+1}}
q_M^{\downarrow}(A)-\sum_{\textstyle{B\subseteq C\atop |M\cap B|\leqslant m-k}}d(B)\,\sum_{\textstyle{A\supseteq B\atop |M\cap A|= m-k}}
q_M^{\uparrow}(A)\, .
\end{eqnarray*}
However, we have
\begin{eqnarray*}
\sum_{\textstyle{A\supseteq B\atop |M\cap A|=m-k+1}} q_M^{\downarrow}(A) &=& \Pr\Big(\exists\, j\in M,\,\exists\, A \supseteq B,\, |M\cap A|=m-k+1:\max_{i\in C\setminus A}T_i<T_j=\min_{i\in A}T_i\Big)\\
&=& \Pr\Big(T_{k:M}\leqslant\min_{i\in B}T_i\Big)
\end{eqnarray*}
and
\begin{eqnarray*}
\sum_{\textstyle{A\supseteq B\atop |M\cap A|=m-k}} q_M^{\uparrow}(A) &=& \Pr\Big(\exists\, j\in M,\,\exists\, A \supseteq B,\, |M\cap A|=m-k:\max_{i\in C\setminus A}T_i=T_j<\min_{i\in A}T_i\Big)\\
&=& \Pr\Big(T_{k:M}<\min_{i\in B}T_i\Big)\, .
\end{eqnarray*}
Thus, we have
\begin{eqnarray*}
p_M^{(k)} &=& \sum_{\textstyle{B\subseteq C\atop |M\cap B|\leqslant m-k+1}}d(B)\,\Pr\Big(T_{k:M}\leqslant\min_{i\in B}T_i\Big)-\sum_{\textstyle{B\subseteq C\atop |M\cap B|\leqslant m-k}}d(B)\,\Pr\Big(T_{k:M}<\min_{i\in B}T_i\Big)\\
&=& \sum_{\textstyle{B\subseteq C\atop |M\cap B|=m-k+1}}d(B)\,\Pr\Big(T_{k:M}\leqslant\min_{i\in B}T_i\Big)+\sum_{\textstyle{B\subseteq C\atop
|M\cap B|\leqslant m-k}}d(B)\,\Pr\Big(T_{k:M}=\min_{i\in B}T_i\Big)\, .
\end{eqnarray*}
We observe that we cannot have $T_{k:M}<\min_{i\in B}T_i$ if $|M\cap B|=m-k+1$. This proves the first formula of the theorem. To see that the
second formula holds, just observe that we cannot have $T_{k:M}=\min_{i\in B}T_i$ if $|M\cap B|>m-k+1$.
\end{proof}

\begin{example}
Consider the structure defined in Example~\ref{ex:7sdfds}. For this structure, we have for instance
\begin{eqnarray*}
p_{\{1,3\}}^{(1)} &=& \sum_{A\subseteq\{1,2,3\}}d(A)\,\Pr\Big(T_{1:\{1,3\}}=\min_{i\in A}T_i\Big)\\
&=& \Pr\Big(\min_{i\in\{1,3\}}T_i=\min_{i\in\{2,3\}}T_i\Big)+1-\Pr\Big(\min_{i\in\{1,3\}}T_i=\min_{i\in\{1,2,3\}}T_i\Big)\\
&=& \big(\Pr(T_3<T_1<T_2)+\Pr(T_3<T_2<T_1)\big)+\big(\Pr(T_2<T_1<T_3)+\Pr(T_2<T_3<T_1)\big).
\end{eqnarray*}
\end{example}

\begin{remark}\label{rem:8s7f}
We observe that the probability $\Pr(T_{k:M}=\min_{i\in A}T_i)$ is exactly the $k$th coordinate of the $M$-signature of the semicoherent system
obtained from the current system by transforming the structure function into $\phi(\bfx)=\prod_{i\in A}x_i$. This result follows immediately
from the fact that the modified system has lifetime $\min_{i\in A}T_i$.
\end{remark}

From Theorem~\ref{thm:sa8d7} we immediately derive the following corollary, which was already established in \cite{MarMat}.

\begin{corollary}
For every $k\in\{1,\ldots,n\}$ and every $j\in C$, we have
$$
p_k ~=~ \sum_{A\subseteq C}d(A)\,\Pr\Big(T_{k:n}=\min_{i\in A}T_i\Big)
$$
and
$$
I_{\mathrm{BP}}^{(j)} ~=~ \sum_{A\subseteq C}d(A)\,\Pr\Big(T_j=\min_{i\in A}T_i\Big)\, .
$$
\end{corollary}

%---------------------------------------------------------------------------------------------- Section 3
\section{Exchangeable component lifetimes}

We now consider the special case when the functions $q_j$ ($j\in C$) satisfy the condition
\begin{equation}\label{eq:8sa6}
q_j(A) ~=~ \frac{1}{n\,{n-1\choose |A|}}\, ,\qquad j\in C\, ,~A\subseteq C\setminus\{j\}\, .
\end{equation}
It is easy to see that this condition holds whenever the lifetimes $T_1,\ldots,T_n$ are i.i.d.\ or, more generally, exchangeable (see
\cite{MarMat}). In this case, for every nonempty subset $M\subseteq C$, we also have
$$
q_M^{\downarrow}(A) ~=~ \frac{|M\cap A|}{n\,{n-1\choose |A|-1}}\, ,\quad q_M^{\uparrow}(A) ~=~ \frac{|M\setminus A|}{n\,{n-1\choose |A|}}\,
,\quad\mbox{and}\quad q(A) ~=~ \frac{1}{{n\choose |A|}}\, .
$$

As mentioned in the introduction, combining this with (\ref{eq:main}) shows that the $M$-signature $\mathbf{p}_M$ does not depend on the distribution function $F$. We then call it \emph{structural $M$-signature} and denoted it by $\mathbf{s}_M$.

Theorem~\ref{thm:8sf6} and Corollaries~\ref{cor:8sdf6} and \ref{cor:8sdf67} are then immediately specialized to the following result.

\begin{corollary}\label{cor:8fsf6}
Assume that the functions $q_j$ ($j\in C$) satisfy condition (\ref{eq:8sa6}). For every nonempty set $M\subseteq C$ and every $k\in\{1,\ldots,m\}$, we have
\begin{equation}\label{eq:fgt1}
s_M^{(k)} ~=~ \sum_{\textstyle{A\subseteq C\atop |M\setminus A|=k}}\,\frac{1}{n\,{n-1\choose |A|}}\,\sum_{j\in M\setminus A}\,\Delta_j\,\phi(A) ~=~
\sum_{j\in M}\,\sum_{\textstyle{A\subseteq C\setminus\{j\}\atop |M\setminus A|=k}}\,\frac{1}{n\,{n-1\choose |A|}}\,\Delta_j\,\phi(A)\, ,
\end{equation}
$$
s_M^{(k)} ~=~ \sum_{j\in M}\,\sum_{\textstyle{A\subseteq C\atop |(M\setminus A)\cup\{j\}|=k}}\, (-1)^{|\{j\}\setminus A|}\, \frac{1}{n\,{n-1\choose
|A\setminus\{j\}|}}\,\phi(A)\, ,
$$
and
$$
s_M^{(k)} ~=~ \sum_{\textstyle{A\subseteq C\atop |M\cap A|=m-k+1}}\, \frac{m-k+1}{n\,{n-1\choose |A|-1}}\,\phi(A)-\sum_{\textstyle{A\subseteq C\atop
|M\cap A|=m-k}}\, \frac{k}{n\,{n-1\choose |A|}}\,\phi(A)\, .
$$
\end{corollary}

From (\ref{eq:fgt1}) we immediately derive new expressions for the structural signature $s_k$, namely
$$
s_k ~=~ \sum_{\textstyle{A\subseteq C\atop |A|=n-k}}\,\frac{1}{n\,{n-1\choose |A|}}\,\sum_{j\in C\setminus A}\,\Delta_j\,\phi(A) ~=~
\sum_{j\in C}\,\sum_{\textstyle{A\subseteq C\setminus\{j\}\atop |A|=n-k}}\,\frac{1}{n\,{n-1\choose |A|}}\,\Delta_j\,\phi(A)\, .
$$

An expression for $s_M^{(k)}$ in terms of the signed domination function is given in the following corollary. Recall first the following well-known identity
\begin{equation}\label{eq:beta}
\int_0^1 t^p(1-t)^q\, dt ~=~ \frac{1}{(p+q+1){p+q\choose p}}{\,},\qquad p,q\in\N.
\end{equation}

\begin{corollary}\label{cor:67f}
Assume that the functions $q_j$ ($j\in C$) satisfy condition (\ref{eq:8sa6}). For every nonempty set $M\subseteq C$ and every $k\in\{1,\ldots,m\}$, we have
$$
s_M^{(k)} ~=~ \sum_{\textstyle{A\subseteq C\atop k-1\leqslant|M\setminus A|\leqslant m-1}}d(A)\,\frac{m-|M\setminus
A|}{k}\,\frac{{|M\setminus A|\choose k-1}}{{|M\setminus A|+|A|\choose k}}\, ,
$$
or equivalently,
$$
s_M^{(k)} ~=~ \sum_{A\subseteq C}d(A)\,\frac{m-|M\setminus A|}{k}\,\frac{{|M\setminus A|\choose k-1}}{{|M\setminus A|+|A|\choose k}}\, .
$$
\end{corollary}

\begin{proof}
For every $A\subseteq C$, let $\phi_A(\bfx)=\prod_{i\in A}x_i$. For every $B\subseteq C$, we then have $\Delta_j\,\phi_A(B)=1$, if $j\in A$ and
$A\setminus\{j\}\subseteq B$, and $0$, otherwise.

Combining Remark~\ref{rem:8s7f} with (\ref{eq:fgt1}), we then obtain
$$
\Pr\Big(T_{k:M}=\min_{i\in A}T_i\Big) ~=~ \sum_{j\in M}\,\sum_{\textstyle{B\subseteq C\setminus\{j\}\atop |M\setminus
B|=k}}\,\frac{1}{n\,{n-1\choose |B|}}\,\Delta_j\,\phi_A(B) ~=~ \sum_{j\in M\cap A}\,\sum_{\textstyle{A\setminus\{j\}\subseteq B\subseteq C\setminus\{j\}\atop |M\setminus B|=k}}\,\frac{1}{n\,{n-1\choose
|B|}}~.
$$
Partitioning $A$ into $A_1=A\cap M$ and $A_2=A\setminus M$ and then using (\ref{eq:beta}) twice and the binomial theorem, the latter expression becomes
\begin{eqnarray*}
\lefteqn{\sum_{j\in A_1}\,\sum_{\textstyle{A_1\setminus\{j\}\subseteq B_1\subseteq M\setminus\{j\}\atop
|B_1|=m-k}}\,\sum_{A_2\subseteq B_2\subseteq C\setminus M}\,\frac{1}{n\,{n-1\choose m-k+|B_2|}}}\\
&=& \sum_{j\in A_1}\,{m-|A_1|\choose m-k-|A_1|+1}\,\sum_{b_2=|A_2|}^{n-m} {n-m-|A_2|\choose b_2-|A_2|}\,\int_0^1 t^{m-k+b_2}\,
(1-t)^{n-1-m+k-b_2}\, dt\\
&=& |A_1|\,{m-|A_1|\choose k-1}\,\int_0^1 t^{m-k+|A_2|}\, (1-t)^{k-1}\, dt ~=~ \frac{|A_1|}{m+|A_2|}\,\frac{{m-|A_1|\choose
k-1}}{{m+|A_2|-1\choose k-1}}\, .
\end{eqnarray*}
We then conclude by Theorem~\ref{thm:sa8d7}.
\end{proof}

From Corollary~\ref{cor:67f} we immediately derive the following result.

\begin{corollary}
Assume that the functions $q_j$ ($j\in C$) satisfy condition (\ref{eq:8sa6}). For every $k\in\{1,\ldots,n\}$ and every $j\in C$, we have
$$
s_k ~= \sum_{\textstyle{A\subseteq C\atop |A|\leqslant n-k+1}}d(A)\,\frac{|A|}{k}\,\frac{{n-|A|\choose k-1}}{{n\choose k}}
$$
and
$$
b_j ~= \sum_{A\subseteq C,\, A\ni j}d(A)\,\frac{1}{|A|}\, .
$$
\end{corollary}

%---------------------------------------------------------------------------------------------- Section 4
\section{Subsignatures associated with modular sets}

It is natural to investigate the concept of $M$-signature in the special case where $M$ is a modular set. In this final section we study this case and show how the $M$-signature is related to the signature of the corresponding module.

Suppose that the system contains a module $(M,\chi)$, where $M\subseteq C$ is the corresponding modular set
and $\chi\colon\{0,1\}^M\to\{0,1\}$ is the corresponding structure function. In this case the structure function of the system expresses through
the composition
\begin{equation}\label{eq:8dsf6}
\phi(\bfx) ~=~ \psi\big(\chi(\bfx^M),\bfx^{C\setminus M}\big)\, ,
\end{equation}
where $\bfx^M=(x_i)_{i\in M}$, $\bfx^{C\setminus M}=(x_i)_{i\in C\setminus M}$. The reduced system (of $n-m+1$ components) obtained from the original system $(C,\phi)$ by considering the modular set $M$ as a single macro-component $[M]$ will be denoted by $(C_M,\psi)$, where $C_M=(C\setminus M)\cup\{[M]\}$ and $\psi\colon\{0,1\}^{C_M}\to\{0,1\}$ is the organizing structure. For general background on modules, see \cite[Chap.~1]{BarPro81}.

As a subsystem, the module $(M,\chi)$ has a lifetime $T_M$, which is defined by
$$
T_M ~=~ \max_{\textstyle{A\subseteq M\atop \chi(A)=1}}\min_{i\in A}T_i{\,}.
$$
Note that $T_M$ is also the lifetime $T_{[M]}$ of component $[M]$ in the reduced system $(C_M,\psi)$. Moreover, it is clear that the event $(T_C=T_j~\mbox{for some}~j\in M)$ coincides with the event $(T_C=T_M)$. From (\ref{eq:probnorm}) it follows that the normalized $M$-signature of the system can then be rewritten as
$$
\frac{p_M^{(k)}}{\sum_{\ell =1}^mp_M^{(\ell)}} ~=~ \Pr(T_C=T_{k:M}\mid T_C=T_M)\, .
$$

The following proposition gives an explicit expression for the probability $\Pr(T_C=T_M)$ in terms of structure $\psi$. We denote by $q_{[M]}^{C_M}\colon 2^{C\setminus M}\to [0,1]$ the relative quality function of component $[M]$ in the reduced system $(C_M,\psi)$. That is,
$$
q_{[M]}^{C_M}(A) ~=~ \Pr\Big(\max_{i\in C\setminus (M\cup A)}T_i<T_{[M]}<\min_{i\in A}T_i\Big),\qquad A\subseteq C\setminus M.
$$

Contrary to functions $q_j$ (which are independent of the structure functions), the function $q_{[M]}^{C_M}$ depends on $T_{[M]}$ and hence on the structure $\chi$ of the module. In particular, it is easy to see that if the components of the module are connected in parallel, then we have
$$
\sum_{j\in M}q_j(A) ~=~ q_{[M]}^{C_M}(A),\qquad A\subseteq C\setminus M.
$$

\begin{proposition}\label{prop:7dsfs}
We have
\begin{equation}\label{eq:7dsfs}
\Pr(T_C=T_M) ~=~ \sum_{A\subseteq C\setminus M} q_{[M]}^{C_M}(A)\,\Delta_{[M]}\psi(A).
\end{equation}
\end{proposition}

\begin{proof}
By definition, the probability $\Pr(T_C=T_M)=\Pr(T_C=T_{[M]})$ is the $[M]$th coordinate of the Barlow-Proschan importance index associated with the reduced system $(C_M,\psi)$. The formula then follows from formula (\ref{eq:a76d5s}).
\end{proof}

\begin{example}\label{ex:as8df76}
Consider a $4$-component system whose structure function is given by
$$
\phi(x_1,x_2,x_3,x_4) ~=~ x_1 (x_2\amalg x_3 x_4) = x_1x_2+x_1x_3x_4-x_1x_2x_3x_4
$$
and consider the module $(M,\chi)$, where $M=\{3,4\}$ and $\chi(x_3,x_4)=x_3x_4$. For such a system we have
$$
\psi(x_{[M]},x_1,x_2) ~=~ x_1x_2+x_1x_{[M]}-x_1x_2x_{[M]}%,\qquad \Delta_{[M]}\psi(x_{[M]},x_1,x_2) ~=~ x_1-x_1x_2{\,},
$$
and by (\ref{eq:7dsfs}) we have
\begin{eqnarray*}
\Pr(T_C=T_M) &=& q_{[M]}^{C_M}(\{1\}) ~=~ \Pr(T_2<T_{[M]}<T_1) ~=~ \Pr(T_2<\min\{T_3,T_4\}<T_1)\\
&=& \Pr(T_2<T_3<T_1<T_4) + \Pr(T_2<T_3<T_4<T_1)\\
&&\null + \Pr(T_2<T_4<T_1<T_3) + \Pr(T_2<T_4<T_3<T_1).
\end{eqnarray*}
\end{example}

Since $(M,\chi)$ is a module, it has its own signature; denote it by $\mathbf{p}^M$. For every $k\in\{1,\ldots,m\}$, the $k$th coordinate of $\mathbf{p}^M$ is given by the probability $p_k^M=\Pr(T_M = T_{k:M})$.

It is not difficult to see that the inclusion $(T_C = T_{k:M})\subset (T_M = T_{k:M})$ holds for every $k\in\{1,\ldots,m\}$. From this observation we derive immediately the identity
\begin{equation}\label{eq:7ds56fs}
p_M^{(k)} ~=~ p_k^M\,\Pr(T_C = T_{k:M}\mid T_M = T_{k:M}).
\end{equation}
This equation shows how the $M$-signature of the system can be related to the signature of the module $(M,\chi)$.

We now show that, under certain assumptions (which are satisfied if the components in $M$ have exchangeable lifetimes), the conditional probability in (\ref{eq:7ds56fs}) can be interpreted as a measure of conditional importance of module $(M,\chi)$.

For every $j\in M$ we denote by $q_j^{M}$ the relative quality function of component $j$ in the module $(M,\chi)$. That is,
$$
q_j^M(A) ~=~ \Pr\Big(\max_{i\in M\setminus A}T_i=T_j<\min_{i\in A}T_i\Big),\qquad A\subseteq M\setminus\{j\}.
$$

We observe that, for every $j\in M$, every $A\subseteq M\setminus\{j\}$ such that $q_j^M(A)\neq 0$, and every $B\subseteq C\setminus M$, we have
\begin{equation}\label{eq:mathExp}
\frac{q_j(A\cup B)}{q_j^M(A)} ~=~ \Pr\Big(\max_{i\in (C\setminus M)\setminus B}T_i<T_j<\min_{i\in B}T_i{\,}\Big|{\,}\max_{i\in M\setminus A}T_i=T_j<\min_{i\in A} T_i\Big).
\end{equation}

\begin{theorem}\label{thm:7sa6df}
Assume that we have
$$
\frac{q_j(A\cup B)}{q_j^M(A)} ~=~ \frac{q_{j'}(A'\cup B)}{q_{j'}^M(A')}
$$
for any $j,j'\in M$, any $A\subseteq M\setminus\{j\}$ and $A'\subseteq M\setminus\{j'\}$, such that $|A|=|A'|$ and $q_j^M(A)\neq 0$ and $q_j^M(A')\neq 0$, and any $B\subseteq C\setminus M$. Then, for any $k\in\{1,\ldots,m\}$, any $j\in M$, and any $A\subseteq M\setminus\{j\}$ such that $|A|=m-k$, we have
\begin{equation}\label{eq:86asf}
p_M^{(k)} ~=~ p_k^M\,\sum_{B\subseteq C\setminus M}\frac{q_j(A\cup B)}{q_j^M(A)}{\,}\Delta_{[M]}\psi(B),
\end{equation}
where the coefficient of $\Delta_{[M]}\psi(B)$ is the conditional probability given in (\ref{eq:mathExp}).
\end{theorem}

\begin{proof}
Let $h_{\phi}\colon [0,1]^C\to\R$, $h_{\chi}\colon [0,1]^M\to\R$, and $h_{\psi}\colon [0,1]^{C_M}\to\R$ be the reliability functions of the structures $\phi$, $\chi$, and $\psi$, respectively. That is,
$$
h_{\phi}(\bfx) ~=~ \sum_{A\subseteq C} \phi(A)\, \prod_{i\in A}x_i\,\prod_{i\in C\setminus A}(1-x_i),\qquad h_{\chi}(\bfx) ~=~ \sum_{A\subseteq M} \chi(A)\, \prod_{i\in A}x_i\,\prod_{i\in M\setminus A}(1-x_i),
$$
and
\begin{equation}\label{eq:psi89}
h_{\psi}(\bfx) ~=~ \sum_{A\subseteq C_M} \psi(A)\, \prod_{i\in A}x_i\,\prod_{i\in C_M\setminus A}(1-x_i).
\end{equation}
By (\ref{eq:8dsf6}) we then have
$$
h_{\phi}(\bfx) ~=~ h_{\psi}\big(h_{\chi}(\bfx^M),\bfx^{C\setminus M}\big)\, .
$$
Using the chain rule it follows that, for every $j\in M$,
\begin{equation}\label{eq:5das7f}
\frac{\partial h_{\phi}}{\partial x_j}{\,}(\bfx) ~=~ \frac{\partial h_{\psi}}{\partial x_{[M]}}{\,}\big(h_{\chi}(\bfx^M),\bfx^{C\setminus M}\big)\, \frac{\partial h_{\chi}}{\partial x_j}{\,}(\bfx^M)\, .
\end{equation}
%Since the function $\partial h_{\psi}/\partial x_{[M]}$ does not depend on its $[M]$-variable, we have
%\begin{equation}\label{eq:5das7f2}
%\frac{\partial h_{\psi}}{\partial x_{[M]}}{\,}\big(h_{\chi}(\bfx^M),\bfx^{C\setminus M}\big) ~=~ \frac{\partial h_{\psi}}{\partial x_{[M]}}{\,}\big(0,\bfx^{C\setminus M}\big)\, .
%\end{equation}
Since any reliability function $h$ is a multilinear polynomial, the partial derivative $\partial h/\partial x_j$ does not depend on variable $x_j$ and coincides with the discrete derivative $\Delta_j h$. From (\ref{eq:5das7f}) it then follows that, for every $A\subseteq M$ and every $B\subseteq C\setminus M$, we have
$$
\Delta_j\,\phi(A\cup B) ~=~ \Delta_{[M]}\,\psi(B)\,\Delta_j\,\chi(A).
$$
Therefore, by (\ref{eq:main}) we obtain
\begin{eqnarray*}
p_M^{(k)} &=& \sum_{\textstyle{A\subseteq M\atop |A|=m-k}}\,\sum_{B\subseteq C\setminus M}\,\sum_{j\in
M\setminus A} q_j(A\cup B)\,\Delta_j\,\phi(A\cup B)\\
&=& \Bigg(\sum_{\textstyle{A\subseteq M\atop |A|=m-k}}\,\sum_{j\in M\setminus
A} q_j^M(A)\,\Delta_j\,\chi(A)\Bigg)\Bigg(\sum_{B\subseteq C\setminus M}\,\frac{q_j(A\cup B)}{q_j^M(A)}\,\Delta_{[M]}\,\psi(B)\Bigg),
\end{eqnarray*}
where the first sum reduces to $p_k^M$ by (\ref{eq:w675re}).
\end{proof}

We can easily observe that the assumptions of Theorem~\ref{thm:7sa6df} hold whenever the component lifetimes are exchangeable. Also, by (\ref{eq:mathExp}) we observe that the sum in (\ref{eq:86asf}) is a mathematical expectation which measures in a sense an importance degree of component $[M]$ in the reduced system $(C_M,\psi)$. Comparing (\ref{eq:7ds56fs}) with (\ref{eq:86asf}) shows that this sum is nothing other than the conditional probability
$$
\Pr(T_C = T_{k:M}\mid T_M = T_{k:M})
$$
whenever it exists (i.e., whenever $p_k^M\neq 0$). Moreover, this sum depends on $k$ but does not depend on the structure $\chi$ of the module. This shows that $p_M^{(k)}$ depends on the structure $\chi$ only through the probability $p_k^M$. In particular, if the components in the module are reorganized so that the probability $p_k^M$ is kept unchanged, then so does the probability $p_M^{(k)}$.

The following result yields integral expressions for the probabilities $\Pr(T_C = T_M)$ and $s_M^{(k)}$ in the exchangeable case.

\begin{corollary}\label{cor:qw87e}
If the function $q_{[M]}^{C_M}$ satisfies condition (\ref{eq:8sa6}), then
\begin{equation}\label{eq:f86dfkj}
\Pr(T_C = T_M) ~=~ \int_0^1 \frac{\partial h_{\psi}}{\partial x_{[M]}}{\,}(t,\ldots,t)\, dt.
\end{equation}
Moreover, if the functions $q_j$ $(j\in M)$ satisfy condition (\ref{eq:8sa6}), then for every $k\in\{1,\ldots,m\}$ we have
\begin{equation}\label{eq:f86dfkj2}
s_M^{(k)} ~=~ s_k^M\,\int_0^1 r_{k,m}(t)\, \frac{\partial h_{\psi}}{\partial x_{[M]}}{\,}(t,\ldots,t)\, dt\, ,
\end{equation}
where $r_{k,m}(t)$ is the p.d.f.\ of the beta distribution on $[0,1]$ with
parameters $\alpha=m-k+1$ and $\beta=k$.
\end{corollary}

\begin{proof}
From (\ref{eq:psi89}) we derive (see \cite{Owe72} for details)
\begin{equation}\label{eq:as65d}
\frac{\partial h_{\psi}}{\partial x_{[M]}}(t,\ldots,t) ~=~ \sum_{B\subseteq C_M\setminus\{[M]\}}t^{|B|}(1-t)^{n-m-|B|}\,\Delta_{[M]}\,\psi(B).
\end{equation}
By (\ref{eq:beta}) we then have
$$
\int_0^1 \frac{\partial h_{\psi}}{\partial x_{[M]}}{\,}(t,\ldots,t)\, dt ~=~ \sum_{B\subseteq C\setminus M}\frac{1}{(n-m+1){n-m\choose |B|}}\,\Delta_{[M]}\,\psi(B){\,},
$$
which, combined with (\ref{eq:7dsfs}), proves (\ref{eq:f86dfkj}).

Let us now prove (\ref{eq:f86dfkj2}). For any $j\in M$, any $A\subseteq M\setminus\{j\}$ such that $|A|=m-k$, and any $B\subseteq C\setminus M$, by (\ref{eq:beta}) we have
$$
\frac{q_j(A\cup B)}{q_j^M(A)} ~=~ \frac{m{m-1\choose m-k}}{n{n-1\choose m-k+|B|}} ~=~ \frac{\int_0^1t^{m-k+|B|}\, (1-t)^{k-1+n-m-|B|}\, dt}{\int_0^1t^{m-k}\,
(1-t)^{k-1}\, dt}{\,}.
$$
Setting $r_{k,m}(t) = t^{m-k}(1-t)^{k-1}/\int_0^1u^{m-k}\, (1-u)^{k-1}\, du$, by (\ref{eq:as65d}) the sum in (\ref{eq:86asf}) then becomes
$$
\int_0^1 r_{k,m}(t)\, \frac{\partial h_{\psi}}{\partial x_{[M]}}{\,}(t,\ldots,t)\, dt
$$
and we can conclude by Theorem~\ref{thm:7sa6df}.
\end{proof}

%---------------------------------------------------------------------------------------------- Acknowledgments
\section*{Acknowledgments}

The author wish to thank Pierre Mathonet for fruitful discussions. This research is supported by the internal research project F1R-MTH-PUL-12RDO2 of the University of Luxembourg.

\end{document}